\documentclass[11pt,american]{amsart}
\usepackage{amsmath,amsfonts,amssymb,amscd,amsthm,amsbsy,bbm,epsf,calc,graphicx,esint,hyperref}
\usepackage{color}
\usepackage{datetime}
\usepackage{mathtools}

\numberwithin{equation}{section}

\newcounter{hours}\newcounter{minutes}

\theoremstyle{plain}
\newtheorem{theorem}{Theorem}[section]

\newtheorem{lemma}[theorem]{Lemma}
\newtheorem{cor}[theorem]{Corollary}
\newtheorem{prop}[theorem]{Proposition}

\newtheorem*{problem}{Problem}

\theoremstyle{definition}                  
\newtheorem{remark}[theorem]{Remark}

\def\dive{\textnormal{div}}

\def\fin{f_{\textnormal{in}}}

\title{Hardy's inequality and (almost) the Landau equation}

\begin{document}
	
\author{Maria Gualdani and Nestor Guillen}	
\thanks{ MPG is supported by DMS-2019335 and would like to
thank NCTS Mathematics Division Taipei for their kind hospitality. } 
\address{The University of Texas at Austin
Mathematics Department 
2515 Speedway Stop C1200
Austin, Texas 78712-1202}
\email{gualdani@math.utexas.edu }
\address{Texas State University
Department of Mathematics
601 University Drive
San Marcos, Texas 78666-4684}
\email{nestor@txstate.edu}

\date{\today}

\begin{abstract}
  In this manuscript we establish an $L^\infty$ estimate for the isotropic analogue of the homogeneous Landau equation. This is done for values of the interaction exponent $\gamma$ in (a part of) the range of very soft potentials. The main observation in our proof is that the classical weighted Hardy inequality leads to a weighted Poincar\'e inequality, which in turn implies the propagation of some $L^p$ norms of solutions. From here, the $L^\infty$ estimate follows from certain weighted Sobolev inequalities and De Giorgi-Nash-Moser theory. 
\end{abstract}

\maketitle

\baselineskip=14pt
\pagestyle{headings}		

\markboth{Hardy's inequality and (almost) the Landau equation}{M. Gualdani, N. Guillen}

\section{Introduction and main result}\label{Introduction}
 
We consider  the equation
\begin{align}   \label{eqn:isotropic equation} 
  \partial_t f  & = c_{d,\gamma} \dive \left (  \int_{\mathbb{R}{^d}}\frac{1}{ |v-v_*|^{-2-\gamma}}\left[ f(v_*,t) \nabla f(v,t) - f(v,t) \nabla_*f(v_*,t)\right]\;dv_*\right ),
\end{align}
where $f(v,t)$ is a non-negative function and the constant $c_{d,\gamma}$ is positive and only depends on the dimension and $\gamma$. 
The constant $\gamma$ belongs to the range of very soft potential, $\gamma \in [-d,-2]$, $d\ge 3$. 
Our chief interest in \eqref{eqn:isotropic equation} comes from its apparent similarity to the homogeneous Landau equation
\begin{align}\label{eqn:Landau}
\partial_t f  = c_{d,\gamma} \dive \left ( \int_{\mathbb{R}{^d}}\frac{\mathbb{P}(v-v_*)}{ |v-v_*|^{-2-\gamma}}\left[ f(v_*,t) \nabla f(v,t) - f(v,t) \nabla_*f(v_*,t)\right]\;dv_*\right ),\end{align}
where $\mathbb{P}(z)$ is the matrix kernel given by the projection into the orthogonal complement to $z$
\begin{align*}
\mathbb{P}(z) = \mathbb{Id} - \frac{z \otimes z }{|z|^2}, \quad z \in  \mathbb{R}^d\setminus \{0\}.
\end{align*}
The term $\mathbb{P}(z)$ in \eqref{eqn:Landau}  is an anisotropy, and it is not present in \eqref{eqn:isotropic equation}.  Therefore one could consider \eqref{eqn:isotropic equation} as an isotropic version of the Landau equation \eqref{eqn:Landau}. 

Equation \eqref{eqn:isotropic equation} was first introduced by Krieger and Strain in \cite{KriStr2012}, motivated precisely by this relationship, with the hope that an understanding of \eqref{eqn:isotropic equation} might lead to new insights on  \eqref{eqn:Landau}.  In \cite{KriStr2012} the authors considered a modification of \eqref{eqn:isotropic equation},  for $\gamma =-d=-3$,
\begin{align}\label{eqn:isotropic equation with alpha}
  \partial_t f = a[f]\Delta f+\alpha f^2,\;\alpha \in [0,1],
\end{align}
where $a[f] :=  f \ast |v|^{-1}$. Here, (\ref{eqn:isotropic equation with alpha}) with $\alpha=1$ corresponds to \eqref{eqn:isotropic equation}. In \cite{KriStr2012}, the authors manage to show that, for $\alpha < 3/4$, spherically symmetric and radially decreasing solutions to \eqref{eqn:isotropic equation with alpha} do not form singularities in finite time, and in fact become smooth. Shortly after, in \cite{GreKriStr2012}, Gressman, Krieger, and Strain improved this result to the range $\alpha<74/75$,  thanks to a new nonlocal inequality for $\gamma=-d$, 
\begin{align*}
  \int_{\mathbb{R}^d}f^{p+1}\;dv \leq \frac{(p+1)^2}{p^2}\int_{\mathbb{R}^d}a[f]|\nabla f^{\frac{p}{2}}|\;dv.
\end{align*}
The authors of the present manuscript showed later in \cite{GuGu15} that for $\alpha=1$ spherically symmetric and radially decreasing solutions to \eqref{eqn:isotropic equation with alpha} regularize instantaneously and remain smooth for all times. Moreover, the approach in \cite{GuGu15} covers \eqref{eqn:isotropic equation with alpha} for all $\alpha \in [0,1]$ and also provides an $L^\infty$ result for spherically symmetric solutions of \eqref{eqn:Landau}, conditioned on a certain spectral bound holding (for more on this, see the discussion about \eqref{eqn:Landau eigenvalue problem} below). Later, in \cite{GuGu16}, we make use of the fact that $a[f]$ is an $A_1$ weight and prove an $L^\infty$-estimate for weak solutions of  \eqref{eqn:isotropic equation}  and \eqref{eqn:Landau}  for $\gamma \geq -2$ and a {{conditional}} $L^\infty$-estimate for solutions of the equation for $\gamma \in [-d,-2)$. 

The aim of this paper is to show a new $L^\infty$-estimate for solutions to \eqref{eqn:isotropic equation} when $\gamma \in (\gamma_*,-2]$, $\gamma_*$ being specified bellow. Here is the main theorem:

\begin{theorem}\label{theorem:isotropic global smooth solutions} 

 Let $f:\mathbb{R}^d\times [0,T] \to \mathbb{R}$ be a non-negative classical solution of \eqref{eqn:isotropic equation} with $\gamma \in \left[ -1-\frac{d}{2},-2\right]$ and initial data $\fin$ that belongs to $L_2^1\cap L^1(\mathbb{R}^d)$.  
 
   \begin{enumerate}
    \item For any $p$ with $1\leq p \leq \frac{d+\gamma}{-2-\gamma}$, the norm $\|f(t)\|_{L^p(\mathbb{R}^d)}$ is non-increasing in time.
	
	\item  Let $\gamma_* \in \left(-1-\frac{d}{2}, -2\right)$ be the unique solution to 
	\begin{align*}
    \tfrac{d}{d+\gamma+2}= \tfrac{d+\gamma}{-\gamma-2}.  
  \end{align*}
If $\fin$ also belongs to $L^p(\mathbb{R}^d)$  for some  $p> \tfrac{d}{d+\gamma+2}$, any solution to  \eqref{eqn:isotropic equation} for $\gamma \in (\gamma_*, -2]$ is uniformly bounded for times away from zero, and 
    \begin{align*}
      \sup \limits_{B_R\times [\tau,T]} f(v,t) \leq C(d,\gamma, \fin,R,\tau,T).
    \end{align*}	
    In particular, for $\gamma \in (\gamma_*,-2]$, classical solutions of \eqref{eqn:isotropic equation} remain smooth for every finite time.
  \end{enumerate}

\end{theorem}

 To illustrate what portion of the ``very soft potential'' range covered in the above theorem, observe that for $d=3$ we have $\gamma_*\approx -2.458$ while for $d=4$ we have $\gamma_*\approx -2.87$. 
The value $\gamma_*$ is uniquely defined in the range $(-d,-2)$.

The proof of Theorem \ref{theorem:isotropic global smooth solutions} is relatively straightforward. For that, we set up first some notation. Let $h[f]$ be the Riesz potential of $f$,
\begin{align*}
h[f]  := (-\Delta)^{-\frac{d+\gamma}{2}}f = C(d,d+\gamma)\int_{\mathbb{R}^d} \frac{f(v_*)}{|v-v_*|^{-\gamma}}\;dv_*,
\end{align*}
where $C(d,d+\gamma)$ is the normalization constant arising in the operator $(-\Delta)^{-\frac{d+\gamma}{2}}$, namely
\begin{align*}
C(d,\alpha) = \pi^{-\tfrac{d}{2}}2^{-\alpha}\frac{\Gamma\left ( \frac{d-\alpha}{2}\right )}{\Gamma\left (\frac{\alpha}{2} \right )}, \quad 0< \alpha < d.
\end{align*}
In particular, $h[f] \to f$ as $\gamma \to -d$. Next, we shall denote by $a[f]$ the convolution 
\begin{align*}
a[f]:= c_{d,\gamma}\int_{\mathbb{R}^d} \frac{f(v_*)}{|v-v_*|^{-2-\gamma}}\;dv_*.
\end{align*}
 With this notation, we rewrite   \eqref{eqn:isotropic equation} as
\begin{align*}
\partial_t f   =  \textrm{div} \left( a[f]\nabla f - f \nabla a[f]\right),
\end{align*}
or, in a non-divergence form, 
\begin{align*}
\partial_t f   =  a[f] \Delta f - (2+\gamma) f h[f],
\end{align*}
provided $c_{d,\gamma}$ is chosen as 
\begin{align*}
   c_{d,\gamma} := \frac{C(d,d+\gamma)}{d+\gamma}, \quad \textrm{or equivalently} \quad   c_{d,\gamma}:= (2+\gamma)C(d,d+2+\gamma),
\end{align*}
so that 
\begin{align}\label{eqn:normalization constant for a}
\Delta a[f]  =  c_{d,\gamma} (d+\gamma)(\gamma+2)\int_{\mathbb{R}^d} \frac{f(v_*)}{|v-v_*|^{-\gamma}}\;dv_* = (2+\gamma)h[f].
\end{align}
One can also define $a[f]$ to be the operator $(2+\gamma)(-\Delta)^{-\frac{d+2+\gamma}{2}}f$. Note that our choice of $c_{d,\gamma}$ (which is well defined for any $\gamma \in [-d,-2]$), does not interfere with the structure of $\eqref{eqn:isotropic equation}$: one can eliminate such constant with a simple time rescaling. 

For $\gamma=-2$, equation \eqref{eqn:isotropic equation} reduces substantially: for sufficiently smooth functions $f$, it simplifies to a heat equation, 
\begin{align*}
  \partial_tf = c_{d,-2} \dive \left (  \int_{\mathbb{R}{^d}} f(v_*,t) \nabla f(v,t)\;dv_*\right ) = c_{d,-2}\|\fin\|_{L^1} \Delta f, 
\end{align*}	  
thanks to the conservation of the $L^1$ norm. This is different from the classical Landau equation \eqref{eqn:Landau} when $\gamma=-2$, which has an additional reaction term arising from the derivatives of $\mathbb{P}(v)$.

Now we explain the basic idea of the proof. For a classical solution of \eqref{eqn:isotropic equation} we have
\begin{align*}
\frac{d}{dt}\int_{\mathbb{R}^d} f^p\;dv = -\frac{4(p-1)}{p}\int_{\mathbb{R}^d}a[f]|\nabla f^{p/2}|^2\;dv + (p-1)(2+\gamma)\int_{\mathbb{R}^d}h[f]f^p\;dv.
\end{align*}
To bound the right hand side, we estimate the second integral by the first one. This is the key point for Theorem \ref{theorem:isotropic global smooth solutions}, and it is achieved through the weighted Hardy inequality \cite{GhoMor2013}
\begin{align*}  
(d+\gamma)^2\int_{\mathbb{R}^d}|v|^{\gamma}\phi^2\;dv \leq 4\int_{\mathbb{R}^d}|v|^{2+\gamma}|\nabla \phi|^2\;dv, \quad \gamma>-d.
\end{align*}
This inequality implies, via convolution, the following weighted Poincare's inequality
\begin{align}\label{eqn:Hardy meets iso Landau_prev}
(d+\gamma)\int_{\mathbb{R}^d} h[f]\phi^2\;dv \leq 4\int_{\mathbb{R}^d}a[f]|\nabla \phi|^2\;dv.
\end{align}
For $\gamma$ belonging to the interval $(\gamma_*,-2]$, thanks to  (\ref{eqn:Hardy meets iso Landau_prev}),  we can show that 
\begin{align*}
\frac{d}{dt}\int_{\mathbb{R}^d} f^p\;dv \le 0,
\end{align*}
for some $p> \frac{d}{d+2+\gamma}$. This is all we need to conclude that $f$ belongs to $L^\infty$ for positive times.

The proof of this last statement exists in the literature. See for example  Theorem 3.8 in \cite{Silvestre2015} or Theorem 2.9 in \cite{GuGu16}. For completeness, we report the proof here, following the method in \cite{GuGu16}.  We will show that a bound for $f$ in  $L^\infty(0,T,L^p(\mathbb{R}^d))$ for some $p> \frac{d}{d+2+\gamma}$ yields an $\varepsilon$-Poincare's inequality of the form 
\begin{align}\label{ineq:eps_prev}
\int_{\mathbb{R}^d}h[f]\;\phi^2\;dv \leq \varepsilon \int_{\mathbb{R}^d}a[f]|\nabla \phi |^2 \;dv+ C(\fin,R)\varepsilon^{-\frac{2}{\eta}}  \int_{\mathbb{R}^d}a[f]\phi^2\;dv.
\end{align}
This last inequality, combined with a Moser's iteration, yields the desired $L^\infty$-bound for $f$.

\subsection{Short background on the homogeneous Landau equation} 

The main result in this note states that for the isotropic equation \eqref{eqn:isotropic equation} there is a nontrivial portion of the ``very soft potentials range'',  $\gamma\in [-d,-2]$, for which one can rule out the formation of singularities in finite time. This is of interest given the similarity of \eqref{eqn:isotropic equation} with the homogeneous Landau equation \eqref{eqn:Landau}, and the unresolved question of $L^\infty$ estimates for the latter equation when $\gamma \in [-d,-2]$. We briefly discuss some of the literature on the Landau equation in order to illustrate what is presently understood about \eqref{eqn:Landau} and what makes the equation in the very soft potentials range much harder to analyze. 

The problem of $L^\infty$-estimates for solutions to the Landau equation \eqref{eqn:Landau} stands as a difficult unresolved question when $\gamma<-2$. Indeed, as $\gamma$ decreases, the kernel in $h[f]$ becomes more and more singular, and one requires more integrability for $f$ to control $h[f]$.
On the contrary, in the range $\gamma \geq -2$, the regularity theory is well developed. For the hard potentials case ($\gamma\geq 0$) the works of Desvillettes and Villani \cite{DesVil2000a,DesVil2000b} do not only answer the question of regularity of solutions, but also cover existence and long time convergence to equilibrium. For soft potentials ($\gamma \in (-2,0)$) and $d=3$, Alexandre, Liao, and Lin \cite{AleLinLia2013} obtain the propagation of $L^2$ estimates for solutions (potentially growing with time), from where higher $L^\infty$ estimates and higher regularity can be obtained. For $\gamma=-2$ we refer to \cite{Wu13}. Silvestre \cite{Silvestre2015} obtains a priori $L^\infty$ estimates for solutions when $\gamma \in (-2,0)$ that are fully determined from the mass, energy, and entropy of the solution (and accordingly are not growing with time). The results in \cite{Silvestre2015} also say that for $\gamma \in [-d,-2)$ the $L^\infty$ norm is controlled once the $L^p$-norm of $f$, with $p> \frac{d}{d+2+\gamma}$, is bounded uniformly in time. Similar results, proven with different method, appear in \cite{GuGu16}. The estimates in \cite{GuGu16} are proved for weak solutions and use the divergence structure of the equation, but yield an $L^\infty$ estimate for $f(v,t)$ that deteriorates as $v$ grows. In contrast, the work \cite{Silvestre2015} uses non-divergence techniques and yields bounds that are global in space.

For very soft potentials there have been several recent works about the nature of potential singularities, which narrow down the ways in which blow up may occur.  In \cite{GGIV2019partial}, Golse, Gualdani, Imbert and Vasseur proved that the set of singular times for weak solutions  to \eqref{eqn:Landau} (with $\gamma=-3$) has Hausdorff dimension at most 1/2. Very recently, Desvillettes, He, and Jiang provided in \cite{DHJ21} new insides about the behavior of solutions  to \eqref{eqn:Landau} (with $\gamma=-3$) in $\dot{H}^1$-norm near the blow-up time. Most importantly, they show that even if there is blow-up, after that solutions become smooth again and remain that way for all later times. In \cite{BeGuSn21}, Bedrossian, Gualdani and Snelson rule out type I self similar blow-up for solutions to \eqref{eqn:Landau} for any $\gamma<-2$. 

There is an important connection between $L^\infty$ bounds and uniqueness. Fournier and {Gu\'erin} proved a uniqueness result for bounded weak solutions in \cite{FournierGuerin2009}, this being for $\gamma \in (-d,-2)$. In fact, the work \cite{FournierGuerin2009} guarantees uniqueness of solutions with $f \in L^1L^\infty$ and in particular to bounded solutions. The work \cite{FournierGuerin2009} was followed by Fournier's work in \cite{Fournier2010} with a corresponding uniqueness result for $\gamma=-d$. Later in \cite{ChGu20}, Chern and Gualdani proved a uniqueness result for sufficiently integrable solutions for the Landau equation with Coulomb interactions. 

\subsection{Outline} The rest of the paper is organized as follows. In Section \ref{sec:Hardy} we recall Hardy's inequality and make the main new observation of the paper, the inequality in Lemma \ref{lem:Hardy meets iso Landau}. The new inequality is immediately put to use in Section \ref{sec:propagating} to prove the first part of Theorem \ref{theorem:isotropic global smooth solutions}. Then, Sections \ref{sec:weighted ineq} and \ref{sec:Moser iteration} deal with weighted normed inequalities and a Moser iteration leading to the $L^\infty$ estimate that makes the second half of Theorem \ref{theorem:isotropic global smooth solutions}.

\section{Hardy's inequality}\label{sec:Hardy}

The main ingredient in this note is the classical Hardy inequality,
\begin{align}\label{eqn:Hardy}
    (d+\gamma)^2\int_{\mathbb{R}^d}|v|^{\gamma}\phi^2\;dv \leq 4\int_{\mathbb{R}^d}|v|^{2+\gamma}|\nabla \phi|^2\;dv, \quad \gamma>-d. 
\end{align}
Let us review one elementary way of proving \eqref{eqn:Hardy}, a deeper and broader discussion on Hardy's inequality can be found in the book by Ghoussoub and Moradifam \cite{GhoMor2013}. First notice that 
\begin{align*}
 -\Delta |v|^{2+\gamma} = (d+\gamma)(-\gamma-2)|v|^\gamma  .
\end{align*}
Multiply both sides of this equation by $\phi$. Integration by parts and Cauchy-Schwarz yield 
\begin{align*}
  (d+\gamma)(-\gamma-2)\int \phi^2|v|^\gamma \;dv &= 2 \int \nabla |v|^{2+\gamma}\phi \nabla \phi \;dv  \\
     &\le 2 \left( \int \alpha^2(v) |\nabla |v|^{2+\gamma}|^2\phi^2\;dv\right)^{1/2}\left( \int \frac{ |\nabla \phi|^2}{\alpha^{2}(v)}\;dv\right)^{1/2}.
\end{align*}
We pick now the best weight $\alpha(v)$ such that 
\begin{align*}
 \alpha^2(v) |\nabla |v|^{2+\gamma}|^2 \le  |v|^{\gamma}.
\end{align*}
or equivalently 
\begin{align*}
 \alpha^2(v)  = \frac{1}{ |v|^{2+\gamma}(2+\gamma)^2}.
\end{align*}
With this choice of $ \alpha^2(v)$, we obtain \eqref{eqn:Hardy}. 
 
\begin{lemma}\label{lem:Hardy meets iso Landau}
  Let $\gamma \in (-d,-2)$. Fix a non-negative $f\in L^1(\mathbb{R}^d)$ and let $a[f]$ and $h[f]$ be as in Section \ref{Introduction}, then the following inequality holds for all $\phi \in C^1_c(\mathbb{R}^d)$ (and limits of such functions)
  \begin{align}\label{eqn:Hardy meets iso Landau}
   (d+\gamma)\int_{\mathbb{R}^d} h[f]\phi^2\;dv \leq 4\int_{\mathbb{R}^d}a[f]|\nabla \phi|^2\;dv.
  \end{align}

\end{lemma}

\begin{proof}
Fix $\phi \in C^1_c(\mathbb{R}^d)$. By a change of variables, we see that \eqref{eqn:Hardy} is equivalent to the inequalities (with $w\in \mathbb{R}^d$)
\begin{align}\label{eqn:Hardy shifted}
  (d+\gamma)^2\int_{\mathbb{R}^d}|v-w|^{\gamma}\phi^2\;dv \leq 4\int_{\mathbb{R}^d}|v-w|^{2+\gamma}|\nabla \phi|^2\;dv.
\end{align}
Let us multiply \eqref{eqn:Hardy shifted} by $f(w)\geq 0$ and integrate the resulting expression in $w$, we obtain
\begin{align}\label{eqn:Hardy convolution}
  & (d+\gamma)^2\int_{\mathbb{R}^d}(f*|v|^\gamma)\phi^2\;dv \leq 4\int_{\mathbb{R}^d}(f*|v|^{2+\gamma} )|\nabla \phi|^2\;dv .
\end{align}
Substituting in \eqref{eqn:Hardy convolution} the expression for $\Delta a[f]$ and making use of \eqref{eqn:normalization constant for a}, the lemma is proved.

\end{proof}

Although elementary, Lemma \ref{lem:Hardy meets iso Landau} is key, as it leads to the propagation of $L^p$ bounds for solutions to \eqref{eqn:isotropic equation}, proven in the next section. The range of $p$'s is limited by the constants appearing in \eqref{eqn:Hardy meets iso Landau}, and this is the sole limitation on the range of $\gamma$'s covered by Theorem \ref{theorem:isotropic global smooth solutions}. This motivates the following (admittedly open ended) eigenvalue problem. 

\begin{problem} Fix $d$ and $\gamma \in [-d,-2]$. Let $f \in L^1(\mathbb{R}^d)$ be non-negative, and let 
\begin{align}\label{eqn:isoLandau eigenvalue problem}
  \Lambda_{\textnormal{iso}}(f) := \inf \limits_{\phi} \frac{\int_{\mathbb{R}^d}a[f]|\nabla \phi|^2\;dv}{\int_{\mathbb{R}^d}h[f]\phi^2\;dv}.
\end{align}
Determine under what circumstances can we say that
\begin{align*}
  \Lambda_{\textnormal{iso}}(f) > \frac{d+\gamma}{4}.
\end{align*}
\end{problem}
If $f$ is just a generic function in $L^1$, then one cannot do better than inequality \eqref{eqn:Hardy meets iso Landau}.  To see this, take a sequence of functions $f_n$ which are converging as $n\to\infty$ to a Dirac delta at $0$. For this sequence, \eqref{eqn:Hardy meets iso Landau} converges to \eqref{eqn:Hardy}, which is known to be sharp \cite[Section 4.3]{GhoMor2013}.

The corresponding problem for the Landau equation would be,
\begin{align}\label{eqn:Landau eigenvalue problem}
  \Lambda_{\textnormal{Landau}}(f) := \inf \limits_{\phi} \frac{\int_{\mathbb{R}^d}(A[f]\nabla \phi,\nabla \phi)\;dv}{\int_{\mathbb{R}^d}h[f]\phi^2\;dv},
\end{align}
where $A[f]:=\int_{\mathbb{R}{^d}}\frac{\mathbb{P}(v-v_*)}{ |v-v_*|^{-2-\gamma}}f(v_*)\;dv_*$. 
The significance of this eigenvalue problem is well known in the Landau and Boltzmann literature. We do not know whether an elementary argument as in Lemma \ref{lem:Hardy meets iso Landau} yields a similar bound for \eqref{eqn:Landau eigenvalue problem}. If one argues by direct analogy with Lemma \ref{lem:Hardy meets iso Landau}, one would have to contend with the projection term $\mathbb{P}(v)$ appearing in $A[f]$, and it is not immediately clear how this can be done. 

The theory of weighted normed inequalities can yield certain estimates for $ \Lambda_{\textnormal{iso}}$, or $\Lambda_{\textnormal{Landau}}$. For instance, the value in \eqref{eqn:isoLandau eigenvalue problem} is directly related to the quantity
\begin{align}\label{eqn:weighted integral quotient}
  \sup \limits_{B} |B|^{\frac{2}{d}}\frac{\int_B h[f]\;dv}{\int_B a[f]\;dv},
\end{align}
which happens to be bounded from above for all non-negative $f \in L^1(\mathbb{R}^d)$ by a universal constant (see \cite{GuGu16} for further discussion). Finally, it is worth mentioning that for any spherically symmetric and radially decreasing $f$ solving \eqref{eqn:Landau} ($\gamma=-d$), the $L^\infty$ norm of $f$ cannot blow up at a finite time $T$ if, for this $f$, the quantity \eqref{eqn:weighted integral quotient} remains bounded by $1/96$ (this is likely a non-sharp estimate). This result was shown in \cite{GuGu15}.

\section{Propagation of $L^p$ bounds}\label{sec:propagating}

In this section we shall make use of Lemma \ref{lem:Hardy meets iso Landau} to show that various $L^p$ norms propagate forward in time, at least for some range of $\gamma$'s.  
\begin{lemma}\label{Lem:propag_iso} (Propagation of integrability.)
Let $f$ be a nonnegative solution to \eqref{eqn:isotropic equation}  with initial data $f(v,0)=\fin$. For every $p$ such that
  \begin{align*}
    1\leq p \leq \frac{d+\gamma}{-2-\gamma}
  \end{align*}
  the norm $\|f(t)\|_{L^p}$ is non-increasing in $t$. In particular, for every $t \in [0,T]$ we have
  \begin{align*}
    \|f(t)\|_{L^p} \leq \|\fin\|_{L^p}.	  
  \end{align*}	  
\end{lemma}

\begin{proof}
Multiply \eqref{eqn:isotropic equation} by $f^{p-1}$ for some $p\geq 1$ and integrate over $\mathbb{R}{^d}$. We obtain
\begin{align*}
\frac{d}{dt}\int_{\mathbb{R}^d} f^p\;dv = -\frac{4(p-1)}{p}\int_{\mathbb{R}^d} a[f] |\nabla f^{p/2}|^2\;dv +(p-1)(-2-\gamma)\int_{\mathbb{R}^d}h[f] f^p\;dv.
\end{align*}
To estimate the last term $\int_{\mathbb{R}^d} h[f]f^{p} \;dv$ we use Lemma \ref{lem:Hardy meets iso Landau} with $\phi = f^{p/2}$. One gets 
\begin{align*}
\frac{d}{dt}\int_{\mathbb{R}^d} f^p\;dv & \leq -4(p-1)\left[\frac{1}{p}-\frac{(-2-\gamma)}{(d+\gamma)}\right]\int_{\mathbb{R}^d} a[f]|\nabla f^{p/2}|^2\;dv.
\end{align*}
It follows that $\|f(t)\|_p^p$ is non-increasing whenever the expression in the brackets is non-positive, which is the case given the assumption on $p$. This concludes the proof of the lemma, and of the first part of Theorem \ref{theorem:isotropic global smooth solutions}.

\end{proof}

\begin{remark}
For there to be any $p$ such that $1\leq p \leq (d+\gamma)/(-2-\gamma)$ it must be that
\begin{align*}
\gamma \ge -1 - \frac{d}{2}.
\end{align*}
It follows that Lemma \ref{Lem:propag_iso} is of no use for values of $\gamma$ close to $-d$.

\end{remark}

\section{Controlling the second moment}
Solutions to \eqref{eqn:isotropic equation} conserve mass and first moment, but not second moment. In the next lemma we show that second moments grow linearly in time, provided $a[f]$ is uniformly bounded. 

\begin{lemma}\label{Lem:propag_moment}
  The second moment of $f$, solution to \eqref{eqn:isotropic equation}, evolves according to the formula
  \begin{align*}
    \frac{d}{dt}\int_{\mathbb{R}^d}f(v,t)|v|^2\;dv = 2(d+2+\gamma)\int_{\mathbb{R}^d}f(v,t)a[f](v,t)\;dv. 
  \end{align*}
  In particular, for all $t\in [0,T]$ and $p> \frac{d}{d+2+\gamma}$ we have
  \begin{align*}
    \int_{\mathbb{R}^d} f(T,v)|v|^2\;dv \leq \int_{\mathbb{R}^d}\fin(v)|v|^2\;dv + TC_{d,\gamma,p}\|f\|_{L^\infty(0,T,L^p)}^{1-\theta}\|\fin\|_{L^1}^{1+\theta},
  \end{align*}
  where $\theta := \frac{|2+\gamma|p'}{d}$.

\end{lemma}	

\begin{proof}
  Integration by parts yields
  \begin{align*}
    \frac{1}{2}\frac{d}{dt}\int_{\mathbb{R}^d} f(t,v)|v|^2\;dv = - \int_{\mathbb{R}^d}(v,a[f]\nabla f-f\nabla a[f])\;dv.
  \end{align*}
  Using the integral form for $a[f]$ we rewrite the expression on the right, leading to
  \begin{align*}
    \frac{1}{2}\frac{d}{dt}\int_{\mathbb{R}^d} f(t,v)|v|^2\;dv =& - c_{d,\gamma}\int_{\mathbb{R}^d}\int_{\mathbb{R}^d}{v}\cdot \frac{f(w)\nabla f(v)-f(v)\nabla f(w)}{|v-w|^{-2-\gamma}}\;dwdv\\
 &   -\frac{ c_{d,\gamma}}{2}\int_{\mathbb{R}^d}\int_{\mathbb{R}^d}{v-w}\cdot \frac{f(w)\nabla f(v)-f(v)\nabla f(w)}{|v-w|^{-2-\gamma}}\;dwdv.
  \end{align*}
  Integration by parts in both $v$ and $w$ yields
  \begin{align*}
   \frac{d}{dt}\int_{\mathbb{R}^d} f(t,v)|v|^2\;dv & =2 (d+2+\gamma)\int_{\mathbb{R}^d} fa[f]\;dv,
  \end{align*}
   since 
  $$
  \textrm{div}\left( \frac{z}{|z|^{-2-\gamma}}\right)= \frac{d+2+\gamma}{|z|^{-2-\gamma}}.
   $$ 
  This proves the first part of the lemma. For the second part, it is clear that
  \begin{align*}
    \frac{1}{2}\frac{d}{dt}\int_{\mathbb{R}^d} f(t,v)|v|^2\;dv & \leq (d+2+\gamma)\|f(t)\|_{L^1} \|a(t)\|_{L^\infty}.
  \end{align*}  
  Then, integrating the resulting inequality in time, the estimate follows thanks to an elementary interpolation argument (see Remark \ref{remark:a Linfinity bound})
  \begin{align*}
    \| a\|_{L^\infty(0,T,L^\infty(\mathbb{R}^d))} \leq C_{d,\gamma,p}\|f(t)\|_{L^\infty(0,T,L^p(\mathbb{R}^d))}^{\theta}\|\fin\|_{L^1}^{1-\theta}.
  \end{align*}
  \end{proof}

\begin{remark}\label{remark:a Linfinity bound}
 The following estimate is well known and we recall it here for completeness: let $p> \frac{d}{d+\gamma+2}$, for every $s>0$ we have that
 \begin{align*}
   a[f]& = c_{d,\gamma}\int_{B_s(v)}f(w)|v-w|^{2+\gamma}\;dw+ c_{d,\gamma}\int_{\mathbb{R}^d\setminus B_s(v)}f(w)|v-w|^{2+\gamma}\;dw\\
  & \leq  c_{d,\gamma}\|f\|_{L^p}\left (\int_{B_s}|w|^{(2+\gamma)p'}\;dw \right )^{1/p'}+ c_{d,\gamma}s^{2+\gamma}\|f\|_{L^1}.
 \end{align*}
 Optimizing the right hand side with respect to $s$, the following estimate follows
 \begin{align*}
   \|a[f]\|_{L^\infty} \leq C_{d,\gamma,p}\|f\|_{L^p}^{\theta}\|f\|_{L^1}^{1-\theta},\;\textnormal{ where } \theta = \frac{|2+\gamma|p'}{d}.
 \end{align*} 
 
 \end{remark}

\begin{cor}\label{cor:lower bound a}
 Let $p > \frac{d}{d+2+\gamma}$. For all $v\in \mathbb{R}^d$ and $t\in [0,T]$, the following inequality holds
  \begin{align*}
    a[f](v,t) \geq \ell  \langle v\rangle^{2+\gamma} ,
  \end{align*}	  
  where
  \begin{align*}	 
    \ell :=  c_{d,\gamma} \frac{\|\fin\|_{L^1}^{-1-\gamma}}{\left( c_1 + c_2T \| f \|_{L^\infty(0,T,L^p(\mathbb{R}^d))}^{1-\theta}\|\fin\|_{L^1}^{1+\theta}\right)^{-2-\gamma}}.
  \end{align*}  
\end{cor}

\begin{proof}
  For any $r>0$ we have	
  \begin{align*}
    \int_{\mathbb{R}^d\setminus B_r}f(v,t)\;dv \leq r^{-2}\int_{\mathbb{R}^d\setminus B_r}f(v,t)|v|^2\;dv \leq r^{-2}\|f(t)\|_{L^1_2}.
  \end{align*}
  From here, taking $r = \frac{2\|f(t)\|_{L^1_2}}{\|f(t)\|_{L^1}}$, we get 
  \begin{align*}
    \int_{B_r}f(v,t)\;dv \geq \tfrac{1}{2}\|\fin\|_{L^1}.
  \end{align*}
   Then, since $|v-w| \leq |v|+r$ whenever $w\in B_{r}$,
  \begin{align*}
    a(v,t) & \geq c_{d,\gamma}\int_{B_r}f(w,t)(|v|+r)^{2+\gamma}\;dw \geq \tfrac{1}{2}c_{d,\gamma} \|\fin\|_{L^1}(|v|+r)^{2+\gamma}.
  \end{align*}
   In particular
  \begin{align*}
    a(v,t) & \geq \tfrac{1}{2}c_{d,\gamma} \|\fin\|_{L^1}r^{2+\gamma}\langle v\rangle^{2+\gamma} \\
    & \geq   c_{d,\gamma} \frac{\|\fin\|_{L^1}^{-1-\gamma}}{\left( c_1 + c_2T \| f(t)\|_{L^p}^{1-\theta}\| \fin \|_{L^1}^{1+\theta}\right)^{-2-\gamma}}\langle v\rangle^{2+\gamma},
  \end{align*}	
  using Lemma \ref{Lem:propag_moment}  to bound $r^{2+\gamma}$ from below. 
\end{proof}

\section{Some weighted inequalities}\label{sec:weighted ineq}

In this section we invoke a few results from the theory of integral inequalities with $A_p$ weights. The result will be integral inequalities of the form
\begin{align*}
  \left ( \int_{\mathbb{R}^d}\omega_1|\phi|^r\;dv \right )^{\frac{1}{r}} \leq C\left ( \int_{\mathbb{R}^d}\omega_2|\nabla \phi|^2\;dv\right )^{\frac{1}{2}},
\end{align*}
for various choices of the exponent $r$, weights $\omega_i$, and constant $C$ which are pertinent to obtaining estimates a la De Giorgi-Nash-Moser for solutions of \eqref{eqn:isotropic equation}. For a more complete discussion we refer the reader to \cite[Section 3.2]{GuGu16}.

A central object in these inequalities is the following product of averages of the weights, taken over an arbitrary cube $Q\subset \mathbb{R}^d$, (here, ``$\fint$'' denotes average over the set of integration)
\begin{align*}
  \sigma_{r,s}(Q,\omega_1,\omega_2) := |Q|^{\frac{1}{d}-\frac{1}{2}+\frac{1}{r}}\left ( \fint_{Q}\omega_1^{s}\;dv \right )^{\frac{1}{rs}}\left( \fint_{Q}\omega_2^{-s}\;dv \right )^{\frac{1}{2s}}.
\end{align*}
The significance of $\sigma_{r,s}(Q,\omega_1,\omega_2)$ is captured by the following theorem (see \cite[Theorem 1]{SawWhe1992}).

\begin{theorem}\label{thm:weighted integral inequalities}
  Let $Q\subset \mathbb{R}^d$, $r\geq 2$, and let $\omega_1,\omega_2$ be two weights. Define, for some $C(d,s,r)$,
  \begin{align*}
    \mathcal{C}_{r,s}(Q,\omega_1,\omega_2) := C(d,s,r) \sup \limits_{Q'\subset 8Q} \sigma_{r,s}(Q',\omega_1,\omega_2).	  
  \end{align*}
  Then, for any $\phi$ supported in $Q$ or any $\phi$ such that $\fint_Q \phi\;dv = 0$, we have
  \begin{align*}
    \left ( \int_{Q}\omega_1 |\phi|^r\;dv\right )^{\frac{1}{r}} \leq \mathcal{C}_{r,s}(Q,\omega_1,\omega_2)\left ( \int_{Q}\omega_2 |\nabla \phi|^2\;dv\right )^{\frac{1}{2}}. 	   
  \end{align*}	   
\end{theorem}

The next two propositions give estimates on $ \sigma_{r,s}(Q,\omega_1,\omega_2)$  for two combination of weights, namely $\omega_1 = a[f]^m, \; \omega_2 = a[f]$ and $\omega_1 = h[f], \;\omega_2 = a[f]$.

There are two exponents that will be appearing repeatedly in what follows: 
\begin{align}\label{eqn:q and m}
  m := \frac{d}{d-2},\;\; q := 2\left (1 + \frac{2}{d} \right ).
\end{align}

\begin{prop}\label{prop:cube averages for Sobolev}
  There exists $s>1$ and $C>0$ depending only on $d$ and $\gamma$ such that for non-negative $f\in L^1(\mathbb{R}^d)$ and any cube $Q\subset \mathbb{R}^d$,
  \begin{align*}
    \sigma_{2m,s}(Q,a[f]^m,a[f]) \leq C.
  \end{align*}
\end{prop}

\begin{proof}
  For $\gamma>-d$, Lemma 3.5 from \cite[Section 3]{GuGu16} says there is some $s>1$ such that
  \begin{align*}
    \left (\fint_{Q}a[f]^{ms}\;dv\right )^{\frac{1}{ms}} \leq C\fint_{Q}a[f]\;dv, \quad C=C(d,\gamma,ms).
  \end{align*}
  As it was also noted in \cite[Section 3]{GuGu16}, there is a universal constant such that
  \begin{align*}
    \fint_{Q}a[f]\;dv \leq C\inf\limits_{Q}a[f],
  \end{align*}	  
  which means also that
  \begin{align*}
    \sup \limits_{Q}a[f]^{-1} \leq C\left ( \fint_{Q}a[f]\;dv \right )^{-1}.	  
  \end{align*}	  
  Putting these two observations together it follows that
  \begin{align*}
    \left(\fint_{Q'} a[f]^{ms}\;dv\right)^{\frac{1}{2ms}}\left(\fint_{Q'} (a[f])^{-s}\;dv\right)^{\frac{1}{2s}}\leq C.
  \end{align*}
  Lastly, $m$ solves $\frac{1}{d}-\frac{1}{2}+\frac{1}{2m} = 0$ (it is its determining property), and the proposition is proved.
\end{proof}
The next one if the key proposition for the proof of (\ref{ineq:eps_prev}):
\begin{prop}\label{prop:cube averages for epsilon Poincare}
  There is $s>\tfrac{d}{2}$ such that given a cube $Q \subset B_R$  with $|Q| \leq 1$ we have
  \begin{align*}
    \left [ \sigma_{2,s}(Q,h[f],a[f])\right ]^2 \leq \theta(|Q|^{\tfrac{1}{d}}) := C(\|f\|_{L^\infty(L^{p(\gamma,d,s)})}, \|f\|_{L^\infty(L^{1})}) \langle R\rangle^{|2+\gamma|} |Q|^{\tfrac{\eta}{d}}.
  \end{align*}
  Here $\eta := 2-\tfrac{d}{s}>0$ and $p(\gamma,d,s):=\tfrac{ds}{d+s(d+\gamma)}$. 
  
  In particular, one can chose $s$ infinitesimally close to $\tfrac{d}{2}$, resulting in $p(\gamma,d,s)$ to be greater, but as close as one wishes to $\tfrac{d}{d+\gamma+2}$.
\end{prop}  

\begin{proof}
  
Classical fractional integral estimates say that
\begin{align*}
  \|h[f]\|_{L^{\frac{dp}{d-p(d+\gamma)}}}\leq C_{d,d+\gamma}\|f\|_{L^p},\; \textnormal{ provided } p<\frac{d}{d+\gamma}.
\end{align*}
We want to choose $p$ so that $s = \frac{dp}{d-p(d+\gamma)}$, which results in $p(\gamma,d,s) = \frac{ds}{d+s(d+\gamma)}$. Therefore,
\begin{align*}
  |Q'|^{\frac{2}{d}}\left ( \fint_{Q'}h[f]^s\;dv \right )^{\frac{1}{s}} \leq C|Q'|^{\frac{2}{d}-\frac{1}{s}}\|f\|_{L^{p(\gamma,d,s)}}.
\end{align*}
We can take $s$ larger but arbitrarily close to $\tfrac{d}{2}$ (to have $\frac{2}{d}-\tfrac{1}{s}$ positive) which results in $p(\gamma,d,s)$ be strictly greater, but arbitrarily close to, $d/(d+2+\gamma)$. Hence,
\begin{align*}
  |Q|^{\frac{2}{d}}\left ( \fint_{Q'}h[f]^s\;dv \right )^{\frac{1}{s}} \leq C|Q'|^{\frac{\eta}{d}} \|f\|_{L^{p(\gamma,d,s)}}.
\end{align*}

Thanks to the bound from below for $a[f]$ from Corollary \ref{cor:lower bound a}, we have
\begin{align*}
  \left(\fint_{Q'} a[f]^{-s}\;dv\right)^{\frac{1}{s}}\leq C\langle R\rangle^{|2+\gamma|},\;\forall\; Q'\subset 8Q.
\end{align*}
We work towards estimating the other factor.

It follows that
\begin{align*}
  |Q'|^{\frac{2}{d}}\left ( \fint_{Q'}h[f]^s\;dv \right )^{\frac{1}{s}}  \left (\fint_{Q'}a[f]^{-s}\;dv \right )^{\frac{1}{s}} \leq C\langle R\rangle^{|2+\gamma|}|Q'|^{\frac{\eta}{d}} \|f\|_{L^{p(\gamma,s)}}.
\end{align*}
This estimate is for all cubes $Q'$ such that $Q'\subset 8Q$, which proves the proposition. 

\end{proof}

An immediate consequence of Theorem \ref{thm:weighted integral inequalities} and Proposition \ref{prop:cube averages for Sobolev} is the following inequality. 
\begin{cor}\label{cor:a_prev_paper_sob}
  There is a universal constant $C$ such that for all $\phi$ we have
  \begin{align*}
\left(  \int_{\mathbb{R}^d}\;a^m[f]\phi^{2m}\;dv\right)^{\frac{1}{m}} \leq C \int_{\mathbb{R}^d}a[f]\;|\nabla \phi|^2  \;dv.
  \end{align*}
\end{cor}
Corollary \ref{cor:a_prev_paper_sob} implies, via an elementary interpolation argument, a space-time integral inequality for functions $\phi:\mathbb{R}^d\times [0,T]\to\mathbb{R}$.
\begin{cor}\label{cor:sob_space_time}
  There is a universal constant $C$ such that
\begin{align}\label{sob_space_time}
  \int_0^T \int_{\mathbb{R}^d}\;a[f]\phi^{q}\;dvdt \leq C\left( \int_0^T \int_{\mathbb{R}^d}a[f]\;|\nabla \phi|^2  \;dvdt+\sup \limits_{(0,T	)}\int_{\mathbb{R}^d}\phi^2\;dv\right)^{q/2}.
\end{align}

\end{cor}

\begin{proof}

We follow the standard proof of this space-time inequality (see proof of Theorem 2.12 and 2.13 in \cite{GuGu16}). First, we estimate the integral of $|\phi|^q$ with weight $a[f]$ by interpolation
\begin{align*}
  \int_{\mathbb{R}^d}a[f]|\phi|^q\;dv & = \int_{\mathbb{R}^d}a[f]\phi^{q\theta+q(1-\theta)}\;dv\\
    & \leq \left (\int_{\mathbb{R}^d}a[f]^{\frac{2m}{(1-\theta)q}}|\phi|^{2m} \;dv \right )^{\frac{(1-\theta)q}{2m}}\left (\int_{\mathbb{R}^d}\phi^2\;dv \right )^{\frac{q\theta}{2}}.
\end{align*}
The exponent $\theta \in (0,1)$ is determined from the relation $\tfrac{1}{q}=\tfrac{1-\theta}{2m}+\tfrac{\theta}{2}$. Simplifying, we obtain
\begin{align*}
  \int_{\mathbb{R}^d}a[f]|\phi|^q\;dv & \leq \left (\int_{\mathbb{R}^d}a[f]^{m}|\phi|^{2m} \;dv \right )^{\frac{1}{m}}\left (\int_{\mathbb{R}^d}\phi^2\;dv \right )^{\frac{m-1}{m}}. 
\end{align*}
Now, by Corollary \ref{cor:a_prev_paper_sob}
\begin{align*}
  \int_{\mathbb{R}^d}a[f]|\phi|^q\;dv \leq C\left (\int_{\mathbb{R}^d}a[f]|\nabla \phi|^2\;dv\right )\left ( \int_{\mathbb{R}^d}\phi^2\;dv \right )^{\frac{m-1}{m}}. 
\end{align*}
Integrating this over time we have
\begin{align*}
  \int_0^T\int_{\mathbb{R}^d}a[f]|\phi|^q\;dvdt \leq C\left ( \int_0^T \int_{\mathbb{R}^d}a[f]|\nabla \phi|^2\;dvdt \right )\left (\sup \limits_{(0,T)}\int_{\mathbb{R}^d}\phi^2\;dv \right )^{\frac{m-1}{m}}.
\end{align*}
From this last inequality it follows trivially that
\begin{align*}
  \int_0^T\int_{\mathbb{R}^d}a[f]|\phi|^q\;dvdt \leq C\left ( \int_0^T \int_{\mathbb{R}^d}a[f]|\nabla \phi|^2\;dvdt + \sup \limits_{(0,T)}\int_{\mathbb{R}^d}\phi^2\;dv \right )^{2-\frac{1}{m}}.
\end{align*}
Noting that $2-\frac{1}{m} = 2-\frac{d-2}{d} = 1+\frac{2}{d} = \frac{q}{2}$, the corollary is proved. 
\end{proof} 

The other important use of Theorem \ref{thm:weighted integral inequalities} is in proving a $\varepsilon$-Poincar\'e inequality, which also relies crucially on Proposition \ref{prop:cube averages for epsilon Poincare} and the $L^p$ bound on $f$.

\begin{cor}\label{cor:Poincare via Lp bound}
 Let $R>0$ and $\varepsilon \in (0,\varepsilon_0)$.  For any $\phi$ supported in $B_R(0)$ we have
\begin{align*}
  \int_{\mathbb{R}^d}h[f]\;\phi^2\;dv \leq \varepsilon \int_{\mathbb{R}^d}a[f]|\nabla \phi |^2 \;dv+ C(\fin,R)\varepsilon^{-\frac{2}{\eta}}  \int_{\mathbb{R}^d}a[f]\phi^2\;dv.
\end{align*}
  Here, $\varepsilon_0 := \theta(1)/C(d,\gamma)$, $\theta(r)$ and $\eta$ are as in Proposition \ref{prop:cube averages for epsilon Poincare}, and
  \begin{align*}
    C(\fin,R) = (4C)^{\frac{2}{\eta}}\langle R\rangle^{\frac{2|2+\gamma|}{\eta}}\|\fin\|_{L^{p(\gamma,d,s)}}^{\frac{2}{\eta}}.	  
  \end{align*}	  
\end{cor}
  
\begin{proof}
  Let $Q$ be any cube in $\mathbb{R}^d$ with $|Q|\leq 1$. Since $\phi = \phi-(\phi)_Q + (\phi)_Q$ it is elementary that
  \begin{align}\label{eqn:Poincare averages inequality}
    \int_{Q} h[f]\phi^2\;dv \leq 4\int_{Q}h[f]( \phi -(\phi)_Q)^2\;dv+4(\phi)_Q^2\int_{Q}h[f]\;dv,
  \end{align}
  where $(\phi)_Q$ denotes the average over $Q$,
  \begin{align*}
    (\phi)_Q = \fint_Q \phi\;dv.
  \end{align*}
  Applying H\"older's inequality to $\int_Q a^{-\frac{1}{2}}(a^{\frac{1}{2}}|\phi|)\;dv$, it follows that
  \begin{align*}
    (\phi)_Q^2 \leq \left ( \fint_Q |\phi|\;dv\right )^2 \leq \left ( \fint_Q a \phi^2\;dv\right ) \left ( \fint_Q a^{-1}\;dv \right ).
  \end{align*}
  Therefore,  
  \begin{align*}
    4(\phi)_Q^2\int_Qh\;dv & \leq 4 \left (\fint_Qh\;dv\right) \left(\fint_Q a^{-1}\;dv \right )\left (\int_Q a\phi^2\;dv\right ) \leq 4C|Q|^{-\frac{2}{d}}\int_{Q}a\phi^2\;dv.
  \end{align*}
  Now, we bound the first term on the right of \eqref{eqn:Poincare averages inequality} by means of Theorem \ref{thm:weighted integral inequalities}, so
  \begin{align*}
    \int_{Q} h \phi^2\;dv & \leq 4\mathcal{C}_{2,2}(Q,h,a)\int_{Q}a|\nabla \phi|^2\;dv+4C|Q|^{-\frac{2}{d}}\int_{Q}a \phi^2\;dv.
  \end{align*}
  Then, by Proposition \ref{prop:cube averages for epsilon Poincare}, we conclude that
  \begin{align*}
    \int_{Q} h \phi^2\;dv & \leq 4C\theta(r)\int_{Q}a|\nabla \phi|^2\;dv+4C|Q|^{-\frac{2}{d}}\int_{Q}a \phi^2\;dv, \;r := |Q|^{\tfrac{1}{d}}.
  \end{align*}
  where $C=C(d,\gamma)$ and $\theta$ is as in Proposition \ref{prop:cube averages for epsilon Poincare}. Adding up these inequalities for each $Q$ of the form $r[0,1]^d +rz$, $z\in\mathbb{Z}^d$, 
  \begin{align*}
    \int_{\mathbb{R}^d}h \phi^2\;dv & \leq 4\theta(r)\int_{\mathbb{R}^d}a|\nabla \phi|^2\;dv+4Cr^{-2} \int_{\mathbb{R}^d}a\phi^2\;dv.
  \end{align*}
  Let $\varepsilon \in (0,\eta(1)/(4C) )$, then there is some $r \in (0,1)$ such that $4\theta(r) = \varepsilon$, namely
  \begin{align*}
    \varepsilon = 4C\langle R\rangle^{|2+\gamma|}\|\fin\|_{L^{p(\gamma,d,s)}}r^\eta.
  \end{align*}
  Indeed, this $r=r(\varepsilon)$ is such that
  \begin{align*}
    r^{-2} = (4C)^{\frac{2}{\eta}}\langle R\rangle^{\frac{2|2+\gamma|}{\eta}}\|\fin\|_{L^{p(\gamma,d,s)}}^{\frac{2}{\eta}}\varepsilon^{-\frac{2}{\eta}}.
  \end{align*}
  Thus, 
  \begin{align*}
    \int_{\mathbb{R}^d}h \phi^2\;dv & \leq \varepsilon\int_{\mathbb{R}^d}a|\nabla \phi|^2\;dv+C(\fin,R)\varepsilon^{-\frac{2}{\eta}} \int_{\mathbb{R}^d}a\phi^2\;dv,
  \end{align*}
  and the corollary is proved. 

\end{proof}

\section{Moser's iteration}\label{sec:Moser iteration}

A $\varepsilon$-Poincare, inequality like the one obtained in Corollary \ref{cor:Poincare via Lp bound}, when valid, yields an $L^\infty$ estimate for the respective solution $f$ of \eqref{eqn:isotropic equation}. This is the case regardless of the value of $\gamma \in [-d,-2]$, as was observed for the Landau equation in \cite{GuGu16}. 

In Proposition \ref{prop:cube averages for epsilon Poincare} and Corollary \ref{cor:Poincare via Lp bound} we have proved that the $\varepsilon$-Poincar\'e inequality holds if $f\in L^\infty(0,T, L^p(\mathbb{R}^d))$ for $p>\frac{d}{d+2+\gamma}$. In view of Lemma \ref{Lem:propag_iso}, solutions to  \eqref{eqn:isotropic equation} belong to $L^p(\mathbb{R}^d)$ with $p>\frac{d}{d+2+\gamma}$ if the initial data belong to the same $L^p(\mathbb{R}^d)$ space and, most importantly, if $\frac{d}{d+2+\gamma} \le \frac{d+\gamma}{-\gamma-2}$. This last inequality holds true for  $\gamma\in (\gamma_*, -2]$, with $\gamma_*$ the unique solution to
$$
\frac{d}{d+2+\gamma} = \frac{d+\gamma}{-\gamma-2}.
$$
Observe that for $\gamma\in(-d,-2)$ the function $\gamma \mapsto d/(d+\gamma+2)$ is strictly decreasing, while $\gamma \mapsto (d+\gamma)/(-\gamma-2)$ is strictly increasing. At $\gamma=-d$ they are equal to $\frac{d}{2}$ and $0$, respectively and at $\gamma=-2$ they are equal to $1$ and $+\infty$, respectively. It follows there is exactly one $\gamma_*\in(-d,-2)$ where they agree. Alternatively, after solving the respective quadratic equation one can see that $\gamma_*$ is given by the formula
\begin{align*}
  \gamma_* = -1-\tfrac{3}{2}d+\tfrac{1}{2}\sqrt{5d^2-4d+4}.
\end{align*}

So, if $\gamma\in (\gamma_*,-2]$, it should be no surprise that the $L^\infty$ estimate for $f$ follows. As mentioned earlier, there are several ways how to show that. We will follow the Moser's approach introduced in \cite{GuGu16}, with consists on estimating the norms 
\begin{align*}
  \left ( \int_0^T\int_{\mathbb{R}^d} a[f]\rho^2 f^{p}\;dvdt \right )^{\frac{1}{p}},
\end{align*}  
for ever increasing powers of $p$ and varying cut off functions $\rho$.  Although the arguments are very similar to those in \cite{GuGu16}, we go in detail over their derivation here. This will complete the proof of Theorem \ref{theorem:isotropic global smooth solutions}.

We start by deriving an energy identity, then we will use the $\varepsilon$-Poincar\'e inequality to control the most problematic term (the integral involving a $h[f]f^p$ term), and arrive at an energy inequality. This energy inequality, together with the space-time weighted inequality (\ref{sob_space_time}), will be repeatedly used to bound the $L^p$ norms as $p \to \infty$. 
 
\begin{prop}\label{prop:energy identity}
Let $p>1$ and let $\rho \in C^2_c(\mathbb{R}^d)$, then
\begin{align*}
  \frac{d}{dt}\int_{\mathbb{R}^d}\rho^2 f^p\;dv & = -\frac{4(p-1)}{p}\int_{\mathbb{R}^d}a[f]|\nabla (\rho f^{p/2})|^2\;dv+\int_{\mathbb{R}^d}(c_1(p)|\nabla \rho|^2- \Delta \rho^2) a[f]f^p\;dv\\
    & \;\;\;\;+(p-1)\int_{\mathbb{R}^d}(-\Delta a[f])\rho^2 f^p\;dv-c_2(p)\int_{\mathbb{R}^d}a[f]f^{p/2}(\nabla (\rho f^{p/2}),\nabla \rho)\;dv,
\end{align*}
where $c_1(p) = 4(p-\frac{(p-1)}{p})$, $c_2(p) = 4(p-2\frac{p-1}{p})$.

\end{prop}

\begin{proof} 
For simplicity we shall write $a$ instead of $a[f]$. From the equation and integration by parts, we have 
\begin{align*} 
  \frac{d}{dt}\int_{\mathbb{R}^d}\rho^2 f^p\;dv & = p \int_{\mathbb{R}^d} \rho^2 f^{p-1}\partial_tf \;dv = -p \int_{\mathbb{R}^d} (\nabla (\rho^2 f^{p-1}),a\nabla f-f\nabla a)\;dv.
\end{align*}
The  integral on the right is equal to the sum of four terms, which we denote $(\textnormal{I}),(\textnormal{II}),(\textnormal{III})$, and $(\textnormal{IV})$, and which we now analyze one by one.

First, note that $(\nabla f^{p-1},\nabla f) = (p-1)f^{p-2}|\nabla f|^2 = (p-1)(4/p^2)|\nabla f^{\frac{p}{2}}|$, therefore
\begin{align*}
  (\textnormal{I}) = \int_{\mathbb{R}^d}\rho^2 (\nabla f^{p-1},a\nabla f)\;dv = \frac{4(p-1)}{p^2}\int_{\mathbb{R}^d}\rho^2 a|\nabla f^{\frac{p}{2}}|^2\;dv.
\end{align*}
Next, we rewrite each of the other three terms using integration by parts, as follows 
\begin{align*}
  (\textnormal{II}) = \int_{\mathbb{R}^d}f^{p-1}(\nabla \rho^2,a\nabla f)\;dv & = \frac{p-1}{p}\int_{\mathbb{R}^d}(\nabla \rho^2,a\nabla f^{p})\;dv\\
    & = -	\frac{p-1}{p}\int_{\mathbb{R}^d}f^p(\nabla \rho^2,\nabla a)+f^pa\Delta \rho^2\;dv,
\end{align*}
\begin{align*} 
  (\textnormal{III}) = \int_{\mathbb{R}^d}\rho^2 (\nabla f^{p-1},-f\nabla a)\;dv & = \frac{p-1}{p}\int_{\mathbb{R}^d}\rho^2 (\nabla f^{p},-\nabla a)\;dv \\ 
   & = \frac{p-1}{p}\int_{\mathbb{R}^d}\rho^2 f^p\Delta a+f^p(\nabla \rho^2,\nabla a)\;dv,
\end{align*} 
\begin{align*}
  (\textnormal{IV}) = \int_{\mathbb{R}^d}f^{p-1}(\nabla \rho^2,-f\nabla a)\;dv & = -\int_{\mathbb{R}^d}f^p(\nabla \rho^2,\nabla a)\;dv \\ 
    & = 2\int_{\mathbb{R}^d} a f^{p/2}(\nabla f^{p/2},\nabla \rho^2)\;dv +\int_{\mathbb{R}^d} a f^p \Delta \rho^2\;dv.
\end{align*}
Adding these identities up, we have
\begin{align*}
  \frac{d}{dt}\int_{\mathbb{R}^d}\rho^2 f^p\;dv & = -\frac{4(p-1)}{p}\int_{\mathbb{R}^d}\rho^2 a|\nabla f^{p/2}|^2\;dv -\int_{\mathbb{R}^d}(\Delta \rho^2) af^p\;dv\\
    & \;\;\;\;+(p-1)\int_{\mathbb{R}^d}(-\Delta a)\rho^2 f^p\;dv-2p\int_{\mathbb{R}^d}af^{p/2}(\nabla f^{p/2},\nabla \rho^2)\;dv.
\end{align*}
We use the elementary identity $\rho \nabla f^{p/2}= \nabla (\rho f^{p/2})-  f^{p/2} \nabla \rho$ and rewrite further,
\begin{align*}
\int_{\mathbb{R}^d} \rho^2 a |\nabla f^{p/2}|^2 \;dv = &\int  a |\nabla (\rho f^{p/2})|^2\;dv + \int  f^p a|\nabla \rho|^2  \;dv \\
&\;\;\;\;-2 \int  f^{p/2}a(\nabla (\rho f^{p/2}),\nabla \rho) \;dv,\\
\int_{\mathbb{R}^d}af^{p/2}(\nabla f^{p/2},\nabla \rho^2)\;dv & = 2\int_{\mathbb{R}^d}af^{p/2}(\nabla (\rho f^{p/2}),\nabla \rho)\;dv-2\int_{\mathbb{R}^d}af^p|\nabla \rho|^2\;dv.
\end{align*}
In conclusion, 
\begin{align*}
  \frac{d}{dt}\int_{\mathbb{R}^d}\rho^2 f^p\;dv & = -\frac{4(p-1)}{p}\int_{\mathbb{R}^d}a|\nabla (\rho f^{p/2})|^2\;dv\\
    & \;\;\;\;+\int_{\mathbb{R}^d}(4(p-\frac{(p-1)}{p})|\nabla \rho|^2- \Delta \rho^2) af^p\;dv\\
    & \;\;\;\;+(p-1)\int_{\mathbb{R}^d}(-\Delta a)\rho^2 f^p\;dv-4(p-2\frac{p-1}{p} )\int_{\mathbb{R}^d}af^{p/2}(\nabla (\rho f^{p/2}),\nabla \rho)\;dv.
\end{align*}

\end{proof}

Since for Theorem \ref{theorem:isotropic global smooth solutions} we only consider $p$'s with $p >  \frac{d}{d+\gamma+2}$, for the rest of this section we will always assume that $p> \frac{d}{d+\gamma+2}$. This has the added benefit of simplifying some of the constants since we are now bounded away from $1$ (note that for $\gamma<-2$, we have $\frac{d}{d+\gamma+2}>1$).
\begin{prop}\label{prop:energy inequality}
  Let $\rho \in C^2_c(B_R)$. Given any three times $T_1<T_2<T_3$ in $[0,T]$ the quantity
  \begin{align}
    \sup \limits_{(T_2,T_3)}\int_{\mathbb{R}^d} \rho^2 f^p\;dv \;&+\;   \frac{(p-1)}{p} \int_{T_2}^{T_3}\int_{\mathbb{R}^d}  a[f] |\nabla (\rho f^{p/2})|^2 \;dvdt \nonumber
  \end{align}
  is not greater than 
  \begin{align} \label{Moser_gen}
    \left(\frac{1}{T_2-T_1}+C_1\right)\int_{T_1}^{T_3}\int \rho^2 f^p\;dvdt + C(d,\gamma)p^2\int_{T_1}^{T_3}\int_{\mathbb{R}^d} a[f]  f^p (|\nabla \rho|^2 +|\Delta \rho^2|) \;dvdt,
  \end{align}
  where $C_1 = C_1(\fin,R,p)$.

\end{prop}

\begin{proof}

Take the identity in Proposition \ref{prop:energy identity}. Per Young's inequality, for every $\varepsilon>0$ we have  
\begin{align*}
  2af^{p/2}|(\nabla (\rho f^{p/2}),\nabla \rho)| \leq \varepsilon a|\nabla (\rho f^{p/2})|^2+ \varepsilon^{-1}af^p|\nabla \rho|^2. 
\end{align*}
For $\varepsilon = \frac{2(p-1)}{p}\frac{1}{2(p-2\frac{p-1}{p})}$ in particular, it follows that
\begin{align*}
  -4(p-2\frac{p-1}{p})af^{p/2}(\nabla (\rho f^{p/2}),\nabla \rho) \leq \frac{2(p-1)}{p}a|\nabla (\rho f^{p/2})|^2+ 2(p-2\frac{p-1}{p})^2\frac{p-1}{p}af^p|\nabla \rho|^2. 
\end{align*}
Therefore, 
\begin{align*}
  \frac{d}{dt}\int_{\mathbb{R}^d}\rho^2 f^p\;dv & \leq -\frac{2(p-1)}{p}\int_{\mathbb{R}^d}a|\nabla (\rho f^{p/2})|^2\;dv\\
    & \;\;\;\;+\int_{\mathbb{R}^d}(4(p-\frac{(p-1)}{p})|\nabla \rho|^2- \Delta \rho^2) af^p\;dv\\
    & \;\;\;\;+(p-1)\int_{\mathbb{R}^d}(-\Delta a)\rho^2 f^p\;dv+2(p-2\frac{p-1}{p})^2\frac{p}{2(p-1)}\int_{\mathbb{R}^d}af^p|\nabla \rho|^2\;dv.
\end{align*}
Combining various terms, this inequality can be written more succinctly as
\begin{align*}
  & \frac{d}{dt}\int_{\mathbb{R}^d}\rho^2 f^p\;dv +\frac{2(p-1)}{p}\int_{\mathbb{R}^d}a|\nabla (\rho f^{p/2})|^2\;dv \\
    & \;\;\;\;\leq \int_{\mathbb{R}^d}(C(p)|\nabla \rho|^2- \Delta \rho^2) af^p\;dv+(p-1)\int_{\mathbb{R}^d}(-\Delta a)\rho^2 f^p\;dv,
\end{align*}
where $C(p) = 2(p-2\frac{p-1}{p})^2\frac{p}{2(p-1)}+4(p-\frac{p-1}{p})$. Since $p\geq \frac{d}{d+\gamma+2}>2$, it is elementary that
\begin{align*}
  C(p) \leq C(d,\gamma)p^2.
\end{align*}
Now we apply Corollary \ref{cor:Poincare via Lp bound} with $\varepsilon = \min\{\frac{1}{p},\eta(1)\}$.
this yields
\begin{align*}
  \frac{d}{dt}\int_{\mathbb{R}^d} \rho^2 f^p\;dv &+ \frac{p-1}{p}  \int_{\mathbb{R}^d}  a |\nabla (\rho f^{p/2})|^2 \;dv \\
  & \leq \int_{\mathbb{R}^d} (C(d,\gamma)p^2|\nabla \rho|^2-\Delta \rho^2)af^p\;dv + C(\fin,\gamma,p_0,R)p^m\int_{\mathbb{R}^d} \rho^2 f^p\;dv .
\end{align*}
Integrate now in $t\in (t_1,t_2)$ and take, on both sides, the sup with respect to $t_2$ in $ (T_2, T_3)$ and average with respect to $t_1$ in $(T_1, T_2)$. From here it results that
\begin{align*}
  \sup \limits_{(T_2,T_3)}\int_{\mathbb{R}^d} \rho^2 f^p\;dv\;+\;   \frac{(p-1)}{p} \int_{T_2}^{T_3}\int_{\mathbb{R}^d} a|\nabla (\rho f^{p/2})|^2 \;dvdt
\end{align*}
is no larger than
\begin{align*}
  \left(\frac{1}{T_2-T_1}+{C(\fin,\gamma,p_0,R)p^m}\right)\int_{T_1}^{T_3}\int_{\mathbb{R}^d} \rho^2 f^p\;dvdt+C(d,\gamma)p^2\int_{T_1}^{T_3}\int_{\mathbb{R}^d} a  f^p (|\nabla \rho|^2+|\Delta \rho^2|) \;dvdt.
\end{align*}

\end{proof}

All that remains of the proof of Theorem \ref{theorem:isotropic global smooth solutions} is covered by the following lemma. This lemma amounts to an $L^p \to L^\infty$ estimate in the spirit of De Giorgi-Nash-Moser theory, following Moser's approach. In what follows, recall that the exponent $q= 2(1+\tfrac{2}{d})$ was defined in \eqref{eqn:q and m} and that $p > \frac{d}{d+\gamma+2}$. 
\begin{lemma}\label{lem:Moser iteration}
  Let $R\geq 1$ and $T>0$, then for any $v\in B_{R}$ and $t\in(T/2,T)$ we have
  \begin{align*}
    f(v,t) \leq C(\fin,p,\gamma,R,T)\left (1+\frac{1}{T} \right )^{\frac{1}{p}\frac{q}{q-2} }\left ( \int_{T/4}^T\int_{B_{2R}}a[f]f^p\;dvdt\right )^{\frac{1}{p}}.
  \end{align*}

\end{lemma}

\begin{proof}
  We introduce the sequences
  \begin{align*}
    T_n := \frac{1}{4}\left ( 2- \frac{1}{2^n} \right )T,\;\;R_n := \left(1 + \frac{1}{2^n}\right)R,\;\; p_n := p \left (\frac{q}{2} \right)^n,\;\;\forall\;n\in\mathbb{N}.
  \end{align*}
  We also pick a sequence of functions $\rho_{n} \in C^2_c(B_{R_n}(0))$ such that
  \begin{align}\label{eqn:iteration cut off functions}	
 \left\{ \begin{array}{ll}
    & 0\leq \rho_n \leq 1\textnormal{ in } B_{R_{n}}(0)\setminus B_{R_{n+1}}(0),\;\; \rho_{n} \equiv 1 \textnormal{ in } B_{R_{n+1}}(0),\\
    & |\nabla \rho_{{n+1}}| \leq C(d)R^{-1}2^n,\\
    & |\Delta \rho^2_{n+1}| \leq C(d)R^{-2}4^n.
 \end{array} \right.
\end{align}
Now, for each $n\ge 0$, let $E_n$ denote the quantity,
\begin{align*}
  E_n := \left ( \int_{T_n}^T\int_{\mathbb{R}^d} \rho_{n}^q a[f]f^{p_n}\;dvdt\right )^{\frac{1}{p_n}}.
\end{align*}
As it is standard for divergence elliptic equations, we are going to derive a recursive relation for $E_n$. To do this, first note $E_{n+1}$ may be written as,
\begin{align*}
  E_{n+1}^{p_{n}} = \left ( \int_{T_{n+1}}^T\int_{\mathbb{R}^d} a[f] \left ( \rho_{{n+1}} f^{\frac{p_{n}}{2}}\right )^{q}\;dvdt \right )^{\frac{2}{q}}.
\end{align*}
Thanks to the space-time inequality \eqref{sob_space_time} we have  
\begin{align*}
  C(d,\gamma)^{-1}E_{n+1}^{p_n} \leq  \sup \limits_{(T_{n+1},T)} \left \{ \int_{\mathbb{R}^d} \rho_{n+1}^2 f^{p_n}(t)\;dv \right \}+ \frac{(p_n-1)}{p_n}\int_{T_{n+1}}^{T}\int_{\mathbb{R}^d} a[f] |\nabla (\rho_{n+1} f^{\frac{p_n}{2}})|^2\;dvdt. 
\end{align*}
Then, the energy inequality from Proposition \ref{prop:energy inequality} says that $C(d,\gamma)^{-1}E_{n+1}^{p_n}$ is no larger than
\begin{align*}
  &\left( \frac{2^{n+2}}{T}+ C(\fin,\gamma,R)p_n^m \right) \int_{T_n}^{T}\int_{\mathbb{R}^d} \rho_{n+1}^2f^{p_n}\;dvdt\\
  & +C(d,\gamma)p_n^2\int_{T_n}^{T}\int_{\mathbb{R}^d}   a[f]f^{p_n}(|\nabla \rho_{n+1}|^2+|\Delta \rho_{n+1}^2|) \;dvdt.
\end{align*}
Note that the first integral is no larger than
\begin{align*}
  \left ( \frac{2^{n+2}}{T}+C(\fin,\gamma,p_0,R)p_n^m \right )\frac{1}{\min\limits_{B_R\times (0,T)} a[f]}\int_{T_n}^T\int_{\mathbb{R}^d}\rho_{n}^q a f^{p_n}\;dv.
\end{align*}
Next, again thanks to $\rho_{n} \equiv 1$ in the support of $\rho_{n+1}$, and in particular
\begin{align*}
  |\nabla \rho_{n+1}|^2 \leq CR^{-2}4^n \rho_{n}^q,\;|\Delta \rho_{n+1}^2| \leq CR^{-2}4^n \rho_{n}^q,
\end{align*}
so the second integral above is no larger than
\begin{align*}
  C(d,\gamma)p_n^2R^{-2}4^n\int_{T_n}^T\int_{\mathbb{R}^d}a[f]\rho_n^qf^{p_n}\;dvdt.
\end{align*}
In conclusion, 
\begin{align*}
  E_{n+1}\leq C\left [ \left (\frac{2^{n+2}}{T}+C(\fin,\gamma,R)p_n^m \right )C(\fin,T)R^{-2-\gamma}  + C(d,\gamma)p_n^2R^{-2}4^n \right ]^{\frac{1}{p_n}}E_n.
\end{align*}

Set
\begin{align*}
  b & = \max\{4q,(q/2)^{m}\},\\
  m & = C\left [ (T^{-1}+C(\fin,\gamma,R)p^m)C(\fin,R^{|2+\gamma|})+C(d,\gamma)p^2R^{-2}\right ].
\end{align*}
Then
\begin{align*}
 E_{n+1}\leq b^{\frac{n}{p_n}}m^{\frac{1}{p_n}}E_n.
\end{align*}
This recursive relation, and a straightforward induction argument, yield that
\begin{align*}
  E_{n} \leq b^{ \sum \limits_{k=0}^{n-1}\frac{k}{p_k}}m^{\sum \limits_{k=0}^{n-1}\frac{1}{p_k} } E_0.
\end{align*}
Since
\begin{align*}
  & \sum \limits_{k=0}^{n-1}\frac{k}{p_k} = \frac{1}{p}\sum \limits_{k=0}^{n-1}k\left ( \frac{2}{q} \right)^k \leq \frac{1}{p}\sum \limits_{k=0}^\infty k \left ( \frac{2}{q}\right )^k = \frac{1}{p}\frac{2q}{(q-2)^2},\\
  &   \sum \limits_{k=0}^{n-1}\frac{1}{p_k} = \frac{1}{p}\sum \limits_{k=0}^{n-1}\left ( \frac{2}{q} \right)^k \leq \frac{1}{p}\sum \limits_{k=0}^\infty \left ( \frac{2}{q}\right )^k = \frac{1}{p}\frac{q}{q-2},
\end{align*}
we conclude that
\begin{align}\label{eqn:Moser iteration Lpn estimate for all n}
  E_n \leq b^{\frac{1}{p}\frac{2q}{(q-2)^2}} m^{\frac{1}{p}\frac{q}{q-2}}E_0.
\end{align}
Observe that
\begin{align*}
  E_0 \leq \left( \int_{T/4}^T \int_{B_{2R}(0)} a[f]f^p\;dvdt \right)^{1/p}.
\end{align*}	
Now, since $\eta_n\geq 1$ in $B_{R}$ and $T_n\leq T/2$ for all $n$, it follows that
\begin{align*}
E_n \geq \left ( \int_{T/2}^T\int_{B_{2R}} a[f]f^{p_n}\;dvdt \right )^{\frac{1}{p_n}}.
\end{align*}  
Considering that $a>0$ everywhere, it follows that 
\begin{align*}
  \limsup \limits_{n \to \infty} E_n \geq \| f\|_{L^\infty(B_{R}\times (T/2,T))}.
\end{align*}
This, together with \eqref{eqn:Moser iteration Lpn estimate for all n} concludes the proof of the lemma, and of Theorem \ref{theorem:isotropic global smooth solutions}.

\end{proof}

\bibliography{landaurefs}
\bibliographystyle{plain}

\end{document}